\documentclass{article}

\usepackage{arxiv}

\usepackage[utf8]{inputenc} 
\usepackage[T1]{fontenc}    
\usepackage{hyperref}       
\usepackage{url}            
\usepackage{booktabs}       
\usepackage{amsfonts}       
\usepackage{nicefrac}       
\usepackage{microtype}      
\usepackage{lipsum}		
\usepackage{graphicx}
\usepackage{natbib}
\usepackage{doi}
\usepackage{amsmath, amsthm, amssymb, algorithm, algpseudocode}
\usepackage[table]{xcolor}
\usepackage{tikz}
\newcommand{\bra}[1]{\left(#1\right)}
\newcommand{\sbra}[1]{\left[#1\right]}

\newcommand{\seqA}[1]{{\left<#1\right>}_{A}}
\newcommand{\vertiii}[1]{{\left\vert\kern-0.25ex\left\vert\kern-0.25ex\left\vert #1
    \right\vert\kern-0.25ex\right\vert\kern-0.25ex\right\vert}}

\newcommand{\normA}[1]{{\left\Vert#1\right\Vert}_{A}}
\newcommand{\abs}[1]{\left\vert#1\right\vert}
\newcommand{\set}[1]{\left\{#1\right\}}

\newcommand{\N}{\mathbb N}

\newcommand {\bh}{\mathcal{B}(\mathcal{H})}

\newcommand {\Sa}{S^{\sharp_{A}}}
\newcommand {\Za}{Z^{\sharp_{A}}}
\newcommand {\Ma}{M^{\sharp_{A}}}

\newcommand {\h}{\mathcal{H}}

\newcommand {\Ta}{T^{\sharp_{A}}}

\renewcommand {\b}{\mathcal{B}}

\newcommand {\tr}{\mathrm{tr}}
\newtheorem{theorem}{Theorem}[section]
\newtheorem{lemma}[theorem]{Lemma}

\newtheorem{corollary}[theorem]{Corollary}

\newtheorem{example}[theorem]{Example}

\newtheorem{remark}[theorem]{Remark}

\newcommand\mystyle{\everymath{\displaystyle}}
\mystyle
\title{Certain Upper Bounds on the $A$-Numerical Radius of Operators in Semi-Hilbertian Spaces and Their Applications}

\author{\href{https://orcid.org/0000-0002-3816-5287}{\includegraphics[scale=0.06]{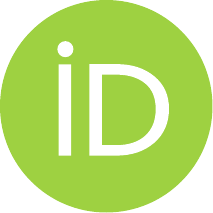}\hspace{1mm}M.H.M.~Rashid}\thanks{Corresponding Author} \\
	Department of Mathematics\&Statistics\\Faculty of Science P.O.Box(7)\\
	Mutah University University\\
	Mutah-Jordan \\
	\texttt{mrash@mutah.edu.jo}
}

\renewcommand{\shorttitle}{\textit{arXiv} Template}

\hypersetup{
pdftitle={Certain Upper Bounds on the $A$-Numerical Radius of Operators in Semi-Hilbertian Spaces and Their Applications},
pdfsubject={q-bio.NC, q-bio.QM},
pdfauthor={M.H.M.Rashid},
pdfkeywords={Heinz mean inequalities, positive semi-definite matrices, Hilbert-Schmidt norm, Young inequality},
}

\begin{document}
\maketitle

\begin{abstract}
	Consider a complex Hilbert space $\left(\mathcal{H}, \langle \cdot, \cdot \rangle\right)$ equipped with a positive bounded linear operator $A$ on $\mathcal{H}$. This induces a semi-norm $\|\cdot\|_A$ through the semi-inner product $\langle x, y \rangle_A = \langle Ax, y \rangle$ for $x, y \in \mathcal{H}$. In this semi-Hilbertian space setting, we investigate the $A$-numerical radius $w_A(T)$ and $A$-operator semi-norm $\|T\|_A$ of bounded operators $T$. This paper presents significant improvements to existing upper bounds for the $A$-numerical radius of operators in semi-Hilbertian spaces. We establish several new inequalities that provide sharper estimates than those currently available in the literature. Notably, we refine the triangle inequality for the $A$-operator semi-norm, offering more precise characterizations of operator behavior. To validate our theoretical advancements, we provide concrete examples that demonstrate the effectiveness of our results. These examples illustrate how our bounds offer tighter estimates compared to previous ones, confirming the practical relevance of our findings. Our work contributes to the broader understanding of numerical radius inequalities and their applications in operator theory, with potential implications for functional analysis and related fields.
\end{abstract}


\section{Introduction}
Let $(\mathcal{H}, \langle \cdot, \cdot \rangle)$ be a nontrivial complex Hilbert space. We denote by $\mathcal{B}(\mathcal{H})$ the $C^*$-algebra of all bounded linear operators on $\mathcal{H}$, and by $I$ the identity operator on $\mathcal{H}$. For any linear subspace $M$ of $\mathcal{H}$, its closure in the norm topology is denoted by $\overline{M}$. Given an operator $T \in \mathcal{B}(\mathcal{H})$, we use the following notation: $\operatorname{ran}(T)$ for the range of $T$, $\ker(T)$ for the null space of $T$ and $T^*$ for the adjoint of $T$

An operator $A \in \mathcal{B}(\mathcal{H})$ is called \textit{positive} if $\langle Ax, x \rangle \geq 0$ for all $x \in \mathcal{H}$, denoted by $A \geq 0$. Every positive operator $A$ induces a positive semidefinite sesquilinear form $\langle \cdot, \cdot \rangle_A : \mathcal{H} \times \mathcal{H} \to \mathbb{C}$ defined by:
\[
\langle x, y \rangle_A = \langle Ax, y \rangle \quad \text{for all } x, y \in \mathcal{H}
\]
This form induces a semi-norm $\|\cdot\|_A$ given by:
\[
\|x\|_A = \sqrt{\langle x, x \rangle_A} = \|A^{1/2}x\|
\]
Note that $|x|_A = 0$ if and only if $x \in \ker(A)$. Consequently, $|\cdot|_A$ is a norm if and only if $A$ is injective, and $(\mathcal{H}, |\cdot|_A)$ is complete if and only if $\operatorname{ran}(A)$ is closed in $\mathcal{H}$.

An operator $S \in \mathcal{B}(\mathcal{H})$ is called an \textit{$A$-adjoint} of $T \in \mathcal{B}(\mathcal{H})$ if:
\[
\langle Tx, y \rangle_A = \langle x, Sy \rangle_A \quad \text{for all } x, y \in \mathcal{H}
\]
Equivalently, the existence of an $A$-adjoint for $T$ corresponds to the solvability of the operator equation $AX = T^*A$. By Douglas' theorem \cite{Douglas}, this equation has a bounded solution if and only if $\operatorname{ran}(T^*A) \subseteq \operatorname{ran}(A)$.

We define the set of all operators admitting $A$-adjoints as:
\[
\mathcal{B}_A(\mathcal{H}) = \{T \in \mathcal{B}(\mathcal{H}) : \operatorname{ran}(T^*A) \subseteq \operatorname{ran}(A)\}
\]
For $T \in \mathcal{B}_A(\mathcal{H})$, the unique solution to $AX = T^A$ is denoted by $T^{\sharp_A}$. Important properties include: $T^{\sharp_A} = A^{\dag}T^A$, where $A^{\dag}$ is the Moore-Penrose inverse of $A$; $\operatorname{ran}(T^{\sharp_A}) \subseteq \operatorname{ran}(A)$; $\ker(T^{\sharp_A}) \subseteq \ker(T^*A)$; $T^{\sharp_A} \in \mathcal{B}_A(\mathcal{H})$; and $(T^{\sharp_A})^{\sharp_A} = P_ATP_A$, where $P_A$ is the orthogonal projection onto $\overline{\operatorname{ran}(A)}$.

Similarly, the set of operators admitting $A^{1/2}$-adjoints is:
\[
\mathcal{B}_{A^{1/2}}(\mathcal{H}) = \{T \in \mathcal{B}(\mathcal{H}) : \|Tx\|_A \leq \lambda\|x\|_A \text{ for all } x \in \mathcal{H}\}
\]
The semi-inner product $\langle \cdot, \cdot \rangle_A$ induces a semi-norm on $\mathcal{B}_{A^{1/2}}(\mathcal{H})$:
\begin{eqnarray*}
  \|T\|_A &=&\sup_{\substack{x \in \overline{\operatorname{ran}(A)} \\ x \neq 0}} \frac{\|Tx\|_A}{\|x\|_A}=
   \sup\{\|Tx\|_A : x \in \mathcal{H}, \|x\|_A = 1\} \\
   &=&\sup\{|\langle Tx, y \rangle_A| : x, y \in \mathcal{H}, \|x\|_A = \|y\|_A = 1\} < \infty
\end{eqnarray*}

Note that $\mathcal{B}_A(\mathcal{H})$ and $\mathcal{B}_{A^{1/2}}(\mathcal{H})$ are two subalgebras of $\mathcal{B}(\mathcal{H})$ and satisfy $\mathcal{B}_A(\mathcal{H}) \subseteq \mathcal{B}_{A^{1/2}}(\mathcal{H})$, see \cite{ACG1, ACG2}. And by the definition of $\|T\|_A$, we can obtain
$$\|T\|_A = \|T^{\sharp_A}\|_A, \quad \|T^{\sharp_A} T\|_A = \|T T^{\sharp_A}\|_A = \|T^{\sharp_A}\|_A^2 = \|T\|_A^2.$$
Moreover, if $T, S \in \mathcal{B}_A(\mathcal{H})$, then $(TS)^{\sharp_A} = S^{\sharp_A} T^{\sharp_A}$, $\|T x\|_A \leq \|T\|_A \|x\|_A$ for any $x \in \mathcal{H}$, and $\|T S\|_A \leq \|T\|_A \|S\|_A$. It also has $\|T + S\|_A \leq \|T\|_A + \|S\|_A$. In particular, an operator $T \in \mathcal{B}(\mathcal{H})$ is $A$-selfadjoint if $AT$ is selfadjoint, which ensures that $\|T\|_A = \sup\{|\langle T x, x \rangle_A| : x \in \mathcal{H}, \|x\|_A = 1\}$, see \cite{Feki}. An operator $T \in \mathcal{B}(\mathcal{H})$ is $A$-positive if $AT$ is positive. Obviously, an $A$-positive operator is always an $A$-selfadjoint operator. And it should be noted that $T^{\sharp_A} T$ and $T T^{\sharp_A}$ are both $A$-positive.

For $A$-bounded operator, $A$-numerical range is defined in \cite{HFA} as
$$W_A(T) := \{|\langle T x, x \rangle_A| : x \in \mathcal{H}, \|x\|_A = 1\}.$$
Also, the $A$-numerical radius is defined by
$$w_A(T) := \sup\{|\lambda| : \lambda \in W_A(T)\}.$$
If $A = I$, $A$-numerical range and $A$-numerical radius will be the classical numerical range and the classical numerical radius. Notice that it may happen that $w_A(T) = \infty$ for some $T \in \mathcal{B}(\mathcal{H})$, see \cite{BN}.

It is well known that if $T\in\b_A(\h)$, then $w_A(T)$ is a semi-norm which is equivalent to
$\normA{T}$, and it holds
\begin{equation}\label{IN-1}
  \frac{1}{2}\normA{T}\leq w_A(T)\leq \normA{T}.
\end{equation}
For $w_A(T)$, an important inequality is the power inequality (see \cite[Proposition 3.10]{MXZ}), which asserts that
\begin{equation}\label{IN-2}
  w_A(T^n)\leq w_A^n(T),\,\,\mbox{for $n\in\N$}.
\end{equation}

Recently, the right-side of the inequality (\ref{IN-1}) has been refined by many researchers. For example,
it has been shown by Zamani in \cite{Zamani-2}, if $T\in\b_A(\h)$, then
\begin{equation}\label{IN-3}
  w_A^2(T)\leq \frac{1}{2}\normA{\Ta T+T\Ta}.
\end{equation}
In \cite{Saddi}, Saddi showed that if $T\in\b_A(\h)$, then
\begin{equation}\label{IN-4}
  w_A^2(T)\leq \frac{1}{2}\bra{\normA{T}^2+w_A(T^2)}.
\end{equation}
Very recently, several refinements of the inequalities in (\ref{IN-1}) have been proved by many
researchers. For example, it was shown by the authors in \cite{Feki-2, Zamani-2} that for $T\in\b_A(\h)$,
then
\begin{equation}\label{IN-5}
  \frac{1}{4}\normA{\Ta T+T\Ta}\leq w_A^2(T)\leq \frac{1}{2}\normA{\Ta T+T\Ta}.
\end{equation}
In \cite{QHB} a refinement of the second inequality in (\ref{IN-5}) was given, which
asserts
\begin{equation}\label{IN-6}
  w_A^4(T)\leq \frac{3}{16}\normA{\Ta T+T\Ta}^2+\frac{1}{8}\normA{\Ta T+T\Ta}w_A(T^2).
\end{equation}
Moreover, a refinement of the second inequality in (\ref{IN-5}) was also proved in \cite{QHC}, which
asserts that if $T\in\b_A(\h)$ and $\beta\geq 0$, then
\begin{equation}\label{IN-7}
  w_A^4(T)\leq \frac{1+2\beta}{16(1+\beta)}\normA{\Ta T+T\Ta}^2+\frac{2\beta+3}{8(1+\beta)}\normA{\Ta T+T\Ta}w_A(T^2)
\end{equation}
and
\begin{equation}\label{IN-8}
  w_A^4(T)\leq \frac{1+2\beta}{8(1+\beta)}\normA{\Ta T+T\Ta}^2+\frac{1}{2(1+\beta)}w_A^2(T^2).
\end{equation}

The $A$-numerical radius and the numerical radius of Hilbert space operators have been extensively studied in recent decades due to their usefulness in analyzing and understanding operator behavior in applications such as numerical analysis, physics, information theory, differential equations, and stability theory of difference approximations. For problems involving hyperbolic initial values, we refer the reader to \cite{Feki, Feki-1, KZ, Saddi} and the references therein. Moreover, the numerical radius is often employed as a more reliable indicator of the convergence rate of iterative algorithms than the spectral radius \cite{BK, Hirz, Kit-S1, Kit-S2}.

For a comprehensive overview of recent results concerning $w_A(\cdot)$, we direct the reader to \cite{FagGor, BF, BFP, Feki, Feki-1, Feki-2, HFA, MXZ, QHB, QHC, Zamani-1, Zamani-2} and the references therein.

The main objective of this work is to improve existing bounds for the $A$-numerical radius of operators in semi-Hilbertian spaces. To gain deeper insight into operator behavior in this setting, we establish novel and sharper inequalities. In particular, we refine the triangle inequality for the $A$-operator seminorm to obtain a more precise characterization of operator norms.

These improvements contribute significantly to the understanding of numerical radius inequalities and their applications. We demonstrate the effectiveness of our proposed bounds through illustrative examples, which also serve to validate our theoretical results. Our findings offer new perspectives and potential applications in operator theory and functional analysis.
\section{Preliminaries and Complementary Results}
To prove the results of this paper, we need the following lemmas. The first lemma
is established in \cite{XYZ}.
\begin{lemma}\label{KR0} Let $x,y,z\in\h$ with $\normA{z}=1$. Then
\begin{equation*}
  \abs{\seqA{x,z}\seqA{z,y}}\leq \frac{1}{2}\bra{\normA{x}\normA{y}+\abs{\seqA{x,y}}}.
\end{equation*}
\end{lemma}
\begin{lemma}\label{Lemma3.18} For every $a,b,e\in\h$ with $\normA{e}=1$, we have
\begin{equation}\label{MD1}
  \abs{\seqA{a,e}\seqA{e,b}}\leq \bra{\frac{1+\alpha}{2}}\normA{a}\normA{b}+\bra{\frac{1-\alpha}{2}}\abs{\seqA{a,b}}
\end{equation}
for every $\alpha\in [0,1]$.
\end{lemma}
\begin{proof} The following is a refinement of the Cauchy-Schwarz inequality
$$\abs{\seqA{a,b}}\leq \abs{\seqA{a,b}-\seqA{a,e}\seqA{e,b}}+\abs{\seqA{a,e}\seqA{e,b}}\leq \normA{a}\normA{b}$$
for every $a,b,e\in\h$ with $\normA{e}=1$.

From inequality (\ref{MD1}), we conclude that
\begin{eqnarray*}
  \abs{\seqA{a,e}\seqA{e,b}} &=& \alpha\abs{\seqA{a,e}\seqA{e,b}}+(1-\alpha)\abs{\seqA{a,e}\seqA{e,b}} \\
   &\leq& \alpha\normA{a}\normA{b}+\frac{1-\alpha}{2}\bra{\normA{a}\normA{b}+\abs{\seqA{a,b}}} \,\,\bra{\mbox{by Lemma \ref{KR0}}} \\
   &=&\bra{\frac{1+\alpha}{2}}\normA{a}\normA{b}+\bra{\frac{1-\alpha}{2}}\abs{\seqA{a,b}}.
\end{eqnarray*}
\end{proof}
Depending on Lemma \ref{Lemma3.18} and the   convexity of the function
$f(t)=t^r,r\geq 1$, we have
\begin{lemma} \label{Lemma3.19}For every $a,b,e\in\h$ with $\normA{e}=1$, we have
\begin{equation}\label{MD2}
  \abs{\seqA{a,e}\seqA{e,b}}^{r}\leq \bra{\frac{1+\alpha}{2}}\normA{a}^{r}\normA{b}^{r}+\bra{\frac{1-\alpha}{2}}\abs{\seqA{a,b}}^{r}
\end{equation}
for every $\alpha\in [0,1]$ and $r\geq 1$.
\end{lemma}

\begin{lemma}\label{Lemma2.4}
For any vectors $a,b,e \in \h$ with $\normA{e} = 1$, the following inequality holds:
\begin{eqnarray}\label{eq:main_inequality}
  \abs{\seqA{a,e}\seqA{e,b}}^2 &\leq&\frac{1}{4}\left[\frac{(2\beta+1)(1+\alpha^2)+2\alpha}{1+\beta}\normA{a}^2\normA{b}^2\right.\nonumber \\
  &+&\left.\frac{(1-\alpha)[2+(1+\alpha)(1+2\beta)]}{1+\beta}\normA{a}\normA{b}\abs{\seqA{a,b}}\right]
\end{eqnarray}
for all $\alpha \in [0,1]$ and $\beta \geq 0$.
\end{lemma}
\begin{proof}
We establish the inequality through the following steps:

1. \textbf{Initial bounds:}
\begin{itemize}
\item From the Cauchy-Schwarz inequality:
\begin{equation}\label{eq:cauchy_schwarz}
\abs{\seqA{a,b}}^2 \leq \normA{a}^2\normA{b}^2
\end{equation}

\item From Lemma \ref{Lemma3.18}:
\begin{eqnarray}\label{eq:lemma318}
  \abs{\seqA{a,e}\seqA{e,b}}^2 &\leq&\left(\frac{1+\alpha}{2}\right)^2\normA{a}^2\normA{b}^2 + \frac{1-\alpha^2}{2}\normA{a}\normA{b}\abs{\seqA{a,b}} \nonumber \\
  &+&\left(\frac{1-\alpha}{2}\right)^2\abs{\seqA{a,b}}^2
\end{eqnarray}
\item The refined inequality:
\begin{equation}\label{eq:refined}
\abs{\seqA{a,b}}^2 \leq \frac{\beta}{1+\beta}\normA{a}^2\normA{b}^2 + \frac{1}{1+\beta}\normA{a}\normA{b}\abs{\seqA{a,b}}
\end{equation}
\end{itemize}

2. \textbf{Combining inequalities:}
\begin{align*}
\abs{\seqA{a,e}\seqA{e,b}}^2
&\leq \left(\frac{1+\alpha}{2}\right)^2\normA{a}^2\normA{b}^2 + \frac{1-\alpha^2}{2}\normA{a}\normA{b}\abs{\seqA{a,b}} \\
&\quad + \left(\frac{1-\alpha}{2}\right)^2\left[\frac{\beta}{1+\beta}\normA{a}^2\normA{b}^2 + \frac{1}{1+\beta}\normA{a}\normA{b}\abs{\seqA{a,b}}\right] \\
&= \left[\frac{(1+\alpha)^2}{4} + \frac{(1-\alpha)^2\beta}{4(1+\beta)}\right]\normA{a}^2\normA{b}^2 \\
&\quad + \left[\frac{1-\alpha^2}{2} + \frac{(1-\alpha)^2}{4(1+\beta)}\right]\normA{a}\normA{b}\abs{\seqA{a,b}}
\end{align*}

3. \textbf{Simplifying coefficients:}
\begin{align*}
C_1 &:= \frac{(1+\alpha)^2}{4} + \frac{(1-\alpha)^2\beta}{4(1+\beta)} = \frac{(2\beta+1)(1+\alpha^2)+2\alpha}{4(1+\beta)} \\
C_2 &:= \frac{1-\alpha^2}{2} + \frac{(1-\alpha)^2}{4(1+\beta)} = \frac{(1-\alpha)[2+(1+\alpha)(1+2\beta)]}{4(1+\beta)}
\end{align*}

4. \textbf{Final form:}
Substituting $C_1$ and $C_2$ yields:
\begin{eqnarray*}
  \abs{\seqA{a,e}\seqA{e,b}}^2 &\leq&\leq \frac{1}{4}\left[\frac{(2\beta+1)(1+\alpha^2)+2\alpha}{1+\beta}\normA{a}^2\normA{b}^2\right. \\
   &+&\left.\frac{(1-\alpha)[2+(1+\alpha)(1+2\beta)]}{1+\beta}\normA{a}\normA{b}\abs{\seqA{a,b}}\right]
\end{eqnarray*}
\end{proof}
\begin{remark}
When $\alpha = 0$ and $\beta = 0$, the inequality simplifies to:
\begin{equation*}
\abs{\seqA{a,e}\seqA{e,b}}^2 \leq \frac{1}{4}\left[\normA{a}^2\normA{b}^2 + 2\normA{a}\normA{b}\abs{\seqA{a,b}}\right]
\end{equation*}
This special case agrees with established results in the literature.
\end{remark}
Depending on the convexity of the function $f(t)=t^r$ $(r\geq 1)$ and Lemma \ref{Lemma2.4}, we have
\begin{corollary}\label{COLLORY-ONE} For every $a,b,e\in\h$ with $\normA{e}=1$, we have
\begin{eqnarray}
  \abs{\seqA{a,e}\seqA{e,b}}^{2r}&\leq& \frac{1}{4}\left[\bra{\frac{(2\beta+1)(1+\alpha^2)+2\alpha}{1+\beta}}\normA{a}^{2r}\normA{b}^{2r}\right.\nonumber \\
  &+&\left.\bra{\frac{\bra{1-\alpha}\sbra{2+(1+\alpha)(1+2\beta)}}{1+\beta}}\normA{a}^{r}\normA{b}^{r}\abs{\seqA{a,b}}^{r}\right]
\end{eqnarray}
for every $\alpha\in [0,1]$, $r\geq 1$ and $\beta\geq 0$.

\end{corollary}
\begin{lemma}\label{Lemma2.5} For every $a,b,e\in\h$ with $\normA{e}=1$, we have
\begin{equation}\label{Major-1}
  \abs{\seqA{a,e}\seqA{e,b}}^{2r}\leq \bra{\frac{1+\alpha+2\beta}{2(1+\beta)}}\normA{a}^{2r}\normA{b}^{2r}+\bra{\frac{1-\alpha}{2(1+\beta)}}\abs{\seqA{a,b}}^{2r}
\end{equation}
for every $\alpha\in [0,1]$, $r\geq 1$  and $\beta\geq 0$.
\end{lemma}
\begin{proof} By similar discussion with Lemma \ref{Lemma2.4}, we have
$$\abs{\seqA{a,b}}^2\leq \abs{\seqA{a,b}}^2+\beta\bra{\normA{a}^2\normA{b}^2-\abs{\seqA{a,b}}^2}. $$
This indicates that
\begin{equation}\label{Major-2}
  \abs{\seqA{a,b}}^2\leq \frac{\beta}{1+\beta}\normA{a}^2\normA{b}^2+\frac{1}{1+\beta}\abs{\seqA{a,b}}^2.
\end{equation}
Applying the convexity of $f(t)=t^r\,(r\geq1)$ on the inequality (\ref{Major-2}) implies that
\begin{equation}\label{Major-3}
  \abs{\seqA{a,b}}^{2r}\leq \frac{\beta}{1+\beta}\normA{a}^{2r}\normA{b}^{2r}+\frac{1}{1+\beta}\abs{\seqA{a,b}}^{2r}.
\end{equation}
On the other hand, combine Lemma \ref{Lemma3.19} and the convexity of the function $f(t)=t^2$, it
can be obtained
\begin{equation}\label{Major-4}
  \abs{\seqA{a,e}\seqA{e,b}}^{2r}\leq \bra{\frac{1+\alpha}{2}}\normA{a}^{2r}\normA{b}^{2r}+\bra{\frac{1-\alpha}{2}}\abs{\seqA{a,b}}^{2r}.
\end{equation}
Consequently, according to the inequalities (\ref{Major-3}) and (\ref{Major-4}), one has
\begin{eqnarray*}
   \abs{\seqA{a,e}\seqA{e,b}}^{2r}&\leq& \bra{\frac{1+\alpha}{2}}\normA{a}^{2r}\normA{b}^{2r} \\
   &+& \bra{\frac{1-\alpha}{2}}\sbra{\frac{\beta}{1+\beta}\normA{a}^{2r}\normA{b}^{2r}+\frac{1}{1+\beta}\abs{\seqA{a,b}}^{2r}}\\
   &\leq&\bra{\frac{1+\alpha+2\beta}{2(1+\beta)}}\normA{a}^{2r}\normA{b}^{2r}+\bra{\frac{1-\alpha}{2(1+\beta)}}\abs{\seqA{a,b}}^{2r}.
\end{eqnarray*}
This completes the proof.
\end{proof}
\begin{remark} Lemma \ref{Lemma2.4} is a refinement of \cite[Lemma 2.3]{QHC} as it may be derived from our result by inputting $\alpha=0$ and $r=1$.
\end{remark}
As a consequence of  Lemma \ref{Lemma2.5}, we have
\begin{corollary}\label{Corollary2.7} For every $a,b,e\in\h$ with $\normA{e}=1$, we have
\begin{equation*}
  \abs{\seqA{a,e}\seqA{e,b}}^{2r}\leq \bra{\frac{1+\alpha+2\beta}{2(1+\beta)}}\normA{a}^{2r}\normA{b}^{2r}+\bra{\frac{1-\alpha}{2(1+\beta)}}\abs{\seqA{a,b}}^{r}\normA{a}^r\normA{b}^r
\end{equation*}
for every $\alpha\in [0,1]$, $r\geq 1$  and $\beta\geq 0$.
\end{corollary}

The following lemma can be found in \cite{QHB}.
\begin{lemma} Let $T\in\b_A(\h)$. Then for any $x,y\in\h$ with $\normA{x}=\normA{y}=1$, we have
$$\abs{\seqA{Tx,y}}^2\leq \sqrt{\seqA{\Ta T x,x}\seqA{T\Ta y,y}}.$$
\end{lemma}
In \cite{QHB}, the authors proved that for any $x,y,z\in\h$,
\begin{equation}\label{max1}
  \abs{\seqA{x,z}}^2+\abs{\seqA{y,z}}^2\leq \normA{z}^2\bra{\max\set{\normA{x}^2,\normA{y}^2}+\abs{\seqA{x,y}}}.
\end{equation}
And in \cite{XYZ}, they established that if $x,y,z\in\h$ with $\normA{z}=1$, then
\begin{equation}\label{max2}
  \abs{\seqA{x,z}\seqA{z,y}}\leq \frac{1}{2}\bra{\normA{x}\normA{y}+\abs{\seqA{x,y}}}.
\end{equation}
\begin{lemma}\label{Qadri1} For every $a,b,e\in\h$ with $\normA{e}=1$, we have
\begin{equation}\label{Result1}
  \abs{\seqA{a,e}}+\abs{\seqA{e,b}}\leq \sqrt{\bra{\normA{a}+\normA{b}}\max\set{\normA{a},\normA{b}}
  +2\abs{\seqA{a,b}}}.
\end{equation}
\end{lemma}
\begin{proof} Utilizing the inequalities (\ref{max1}) and (\ref{max2}), we deduce that
\begin{eqnarray*}
  \bra{\abs{\seqA{a,e}}+\abs{\seqA{e,b}}}^2 &=&\abs{\seqA{a,e}}^2+\abs{\seqA{e,b}}^2+2\abs{\seqA{a,e}}\abs{\seqA{e,b}} \\
   &\leq&\bra{\max\set{\normA{a}^2,\normA{b}^2}+\abs{\seqA{a,b}}}+  \bra{\normA{a}\normA{b}+\abs{\seqA{a,b}}}\\
   &=& \bra{\normA{a}+\normA{b}}\max\set{\normA{a},\normA{b}}
  +2\abs{\seqA{a,b}}.
\end{eqnarray*}
\end{proof}
Combining Lemma \ref{Lemma3.18} and \cite[Lemma 2.5]{QHB}, we have
\begin{lemma}\label{Comb1} Let $a,b,e\in\h$ with $\normA{e}=1$. Then
\begin{eqnarray}\label{COmb-A1}
 \bra{\abs{\seqA{a,e}}+\abs{\seqA{e,b}}}^2 &\leq&\sqrt{\abs{\seqA{a,a}}^2+\abs{\seqA{b,b}}^2+2\abs{\seqA{a,b}}^2}\nonumber  \\
   &+& (1+\alpha)\normA{a}\normA{b}+(1-\alpha)\abs{\seqA{a,b}}
\end{eqnarray}
 for every $\alpha\in [0,1]$.
\end{lemma}
\section{Some upper bounds of $A$-numerical radius}
The main goal of this section is to derive several upper bounds for $A$-numerical
radius which are refinements of some existing ones.
\begin{theorem}\label{Thm-3.1} Let $T\in\b_A(\h)$, $\alpha\in[0,1]$ and $\beta\geq 0$. Then
\begin{equation}\label{Ineq3.1}
  w_A^4(T)\leq \frac{\gamma_1}{16}\normA{\Ta T+T\Ta}^2+ \frac{\gamma_2}{8}\normA{\Ta T+T\Ta}w_A(T^2),
\end{equation}
where $\gamma_1=\frac{(2\beta+1)(1+\alpha^2)+2\alpha}{1+\beta}$ and $\gamma_2=\frac{\bra{1-\alpha}\sbra{2+(1+\alpha)(1+2\beta)}}{1+\beta}$.
\end{theorem}
\begin{proof} Let $x\in\h$ with $\normA{x}=1$. Then
\begin{eqnarray*}
  \abs{\seqA{Tx,x}}^4 &=& \abs{\seqA{Tx,x}\seqA{x,\Ta x}}^2 \\
   &\leq& \frac{\gamma_1}{4}\normA{Tx}^2\normA{\Ta x}^2+\frac{\gamma_2}{4}\normA{Tx}\normA{\Ta x}\abs{\seqA{Tx,\Ta x}}\\
  &&\bra{\mbox{by Lemma \ref{Lemma2.4}}}\\
  &=&\frac{\gamma_1}{4}\bra{\sqrt{\seqA{\Ta Tx,x}\seqA{T\Ta x,x}}}^2+\frac{\gamma_2}{4}\sqrt{\seqA{\Ta Tx,x}\seqA{T\Ta x,x}}\abs{\seqA{T^2x,x}}\\
  &\leq&\frac{\gamma_1}{16}\seqA{\bra{\Ta T+T\Ta}x,x}^2+ \frac{\gamma_2}{8}\seqA{\bra{\Ta T+T\Ta}x,x}\abs{\seqA{T^2x,x}}\\
  &&\bra{\mbox{by the arithmetic-geometric mean inequality}}\\
  &\leq&\frac{\gamma_1}{16}\normA{\Ta T+T\Ta}^2+ \frac{\gamma_2}{8}\normA{\Ta T+T\Ta}w_A(T^2).
\end{eqnarray*}
Taking the supremum over all vectors $x\in\h$ with $\normA{x}=1$ in the above inequality, we deduce that
$$w_A^4(T)\leq \frac{\gamma_1}{16}\normA{\Ta T+T\Ta}^2+ \frac{\gamma_2}{8}\normA{\Ta T+T\Ta}w_A(T^2).$$
This completes the proof.
\end{proof}
\begin{remark} The inequality (\ref{IN-7}) may be derived by putting $\alpha=0$ in Theorem \ref{Thm-3.1}, hence Theorem \ref{Thm-3.1} is a refinement of that inequality.
\end{remark}
To show that Theorem \ref{Thm-3.1} is a nontrivial improvement of the inequalities  (\ref{IN-6}), we give the following example.
\begin{example} Let $T=\begin{bmatrix} 1& 0 \\ 1 & 1 \\ \end{bmatrix}$, $A=\begin{bmatrix} 1 & -1 \\ -1 & 2 \\ \end{bmatrix}$,
$\alpha=0.5$ and $\beta=1$. Then elementary calculations show that
$\Ta=\begin{bmatrix} -1& 4 \\  -1&  3\\ \end{bmatrix}$, $\Ta T+T\Ta=\begin{bmatrix} 2& 8 \\  0&  10\\ \end{bmatrix}$,
$$\normA{\Ta T+T\Ta}\approx 12.385, w_A(T)=2, w_A(T^2)\approx 2.3, \gamma_1=\frac{19}{8}\,\,\mbox{and}\,\, \gamma_2=\frac{13}{8}.$$
Hence, the right-side of inequality (\ref{Ineq3.1}) gives
$$\sqrt[4]{\frac{\gamma_1}{16}\normA{\Ta T+T\Ta}^2+ \frac{\gamma_2}{8}\normA{\Ta T+T\Ta}w_A(T^2)}\approx 2.31$$
While the right-side of  (\ref{IN-6})gives
$$\sqrt[4]{\frac{3}{16}\normA{\Ta T+T\Ta}^2+\frac{1}{8}\normA{\Ta T+T\Ta}w_A(T^2)}\approx 2.39.$$
Moreover,
$$2=w_A(T)\leq \sqrt[4]{\frac{\gamma_1}{16}\normA{\Ta T+T\Ta}^2+ \frac{\gamma_2}{8}\normA{\Ta T+T\Ta}w_A(T^2)}\approx 2.31.$$
So, the example also illustrate the validity of Theorem \ref{Thm-3.1}.
\end{example}
\begin{theorem}\label{Thm3.2} Let $T\in\b_A(\h)$, $\alpha\in[0,1]$, $r\geq 1$ and $\beta\geq 0$. Then
\begin{equation}\label{Ineq3.2}
  w_A^{4r}\leq \frac{\delta_1}{4}\normA{\bra{\Ta T}^{r}+\bra{T\Ta}^{r}}^2+\delta_2w_A^{2r}(T^2),
\end{equation}
where $\delta_1=\bra{\frac{1+\alpha+2\beta}{2(1+\beta)}}$ and $\delta_2=\bra{\frac{1-\alpha}{2(1+\beta)}}$.
\end{theorem}
\begin{proof} Let $x\in\h$ with $\normA{x}=1$. Then
\begin{eqnarray*}
  \abs{\seqA{Tx,x}}^{4r} &=& \abs{\seqA{Tx,x}\seqA{x,\Ta x}}^{2r}  \\
   &\leq& \delta_1\normA{Tx}^{2r}\normA{\Ta x}^{2r}+\delta_2\abs{\seqA{Tx,\Ta x}}^{2r}\\
   &&\bra{\mbox{by Lemma \ref{Lemma2.5}}}\\
   &\leq& \delta_1\bra{\sqrt{\seqA{\Ta Tx,x}^{r}\seqA{T\Ta x,x}^{r}}}^2+\delta_2\abs{\seqA{T^2x, x}}^{2r}\\
   &\leq&\delta_1\bra{\sqrt{\seqA{\bra{\Ta T}^{r}x,x}\seqA{\bra{T\Ta}^{r} x,x}}}^2+\delta_2\abs{\seqA{T^2x, x}}^{2r}\\
   &&\bra{\mbox{by the H\"older McCarthy inequality}}\\
   &\leq& \frac{\delta_1}{4}\seqA{\bra{\bra{\Ta T}^{r}+\bra{T\Ta}^{r}}x,x}^2 +\delta_2\abs{\seqA{T^2x, x}}^{2r}\\
   &&\bra{\mbox{by the arithmetic-geometric mean inequality}}\\
   &\leq& \frac{\delta_1}{4}\normA{\bra{\Ta T}^{r}+\bra{T\Ta}^{r}}^2+\delta_2w_A^{2r}(T^2).
\end{eqnarray*}
Taking the supremum over all vectors $x\in\h$ with $\normA{x}=1$ in the above inequality, we deduce that
$$w_A^{4r}\leq \frac{\delta_1}{4}\normA{\bra{\Ta T}^{r}+\bra{T\Ta}^{r}}^2+\delta_2w_A^{2r}(T^2).$$
This completes the proof.
\end{proof}
\begin{remark}(i) If we put $r=1$ and $\alpha=0$, we obtain the inequality (\ref{IN-8}). So, Theorem \ref{Thm3.2}
is a refinement of that inequality. \\
(ii) If we set $r=1$ in the inequality (\ref{Ineq3.2}), we deduce that the inequality (\ref{Ineq3.2}) is sharper than the inequality
(\ref{IN-3}). Indeed,
\begin{eqnarray*}
  w_A^{4}(T)&\leq& \bra{\frac{1+\alpha+2\beta}{8(1+\beta)}}\normA{\Ta T+T\Ta}^2+\bra{\frac{1-\alpha}{2(1+\beta)}}w_A^{2}(T^2) \\
   &\leq& \bra{\frac{1+\alpha+2\beta}{8(1+\beta)}}\normA{\Ta T+T\Ta}^2+\bra{\frac{1-\alpha}{2(1+\beta)}}w_A^{4}(T)\\
   &\leq& \bra{\frac{1+\alpha+2\beta}{8(1+\beta)}}\normA{\Ta T+T\Ta}^2+\bra{\frac{1-\alpha}{8(1+\beta)}}\normA{\Ta T+T\Ta}^2\\
   &=& \frac{1}{4}\normA{\Ta T+T\Ta}^2.
\end{eqnarray*}
\end{remark}
Now, we give an example to show that Theorem \ref{Thm3.2} is a nontrivial improvement
of inequality (\ref{IN-3}).
\begin{example}Let $T=\begin{bmatrix} 1& 0 \\ 1 & 1 \\ \end{bmatrix}$, $A=\begin{bmatrix} 1 & -1 \\ -1 & 2 \\ \end{bmatrix}$,
$\alpha=0.5$, $r=1$ and $\beta=1$. Then elementary calculations show that
$\Ta=\begin{bmatrix} -1& 4 \\  -1&  3\\ \end{bmatrix}$, $\Ta T+T\Ta=\begin{bmatrix} 2& 8 \\  0&  10\\ \end{bmatrix}$,
$$\normA{\Ta T+T\Ta}\approx 12.385, w_A(T)=2, w_A^2(T^2)\approx 5.29, \delta_1=0.875\,\,\mbox{and}\,\, \delta_2=0.125.$$
  Hence,
  $$\bra{\frac{1+\alpha+2\beta}{8(1+\beta)}}\normA{\Ta T+T\Ta}^2+\bra{\frac{1-\alpha}{2(1+\beta)}}w_A^{2}(T^2)\approx 34.21$$
  and
  $$\frac{1}{4}\normA{\Ta T+T\Ta}^2\approx 38.34.$$
  Consequently,
  $$\bra{\frac{1+\alpha+2\beta}{8(1+\beta)}}\normA{\Ta T+T\Ta}^2+\bra{\frac{1-\alpha}{2(1+\beta)}}w_A^{2}(T^2)<\frac{1}{4}\normA{\Ta T+T\Ta}^2.$$
  Moreover, we can use this example to illustrate the validity of Theorem \ref{Thm3.2}, since
  $$2\approx w_A(T)\leq \sqrt[4]{\bra{\frac{1+\alpha+2\beta}{8(1+\beta)}}\normA{\Ta T+T\Ta}^2+\bra{\frac{1-\alpha}{2(1+\beta)}}w_A^{2}(T^2)}\approx
  2.41.$$
\end{example}
\begin{remark} If we taking $\alpha=0,\beta=0$ and $r=1$ in Theorem \ref{Thm3.2}, we will obtain that
$$w_A^{4}(T)\leq \frac{1}{8}\normA{\Ta T+T\Ta}^2+\frac{1}{8}w_A^2(T^2)$$
which is sharper than the inequality (\ref{IN-6}). Indeed,
\begin{eqnarray*}
  \frac{1}{8}\normA{\Ta T+T\Ta}^2+\frac{1}{8}w_A^2(T^2) &\leq& \frac{1}{8}\normA{\Ta T+T\Ta}^2+\frac{1}{4}w_A^4+\frac{1}{4}w_A^2(T)w_A(T^2)\\
   &\leq& \frac{1}{8}\normA{\Ta T+T\Ta}^2+\frac{1}{16}\normA{\Ta T+T\Ta}^2\\
   &+&\frac{1}{8}\normA{\Ta T+T\Ta}^2w_A(T^2)\\
   &=&\frac{3}{16}\normA{\Ta T+T\Ta}^2+\frac{1}{8}\normA{\Ta T+T\Ta}^2w_A(T^2).
\end{eqnarray*}
Therefore,
$$w_A^{4}(T)\leq \frac{1}{8}\normA{\Ta T+T\Ta}^2+\frac{1}{8}w_A^2(T^2)\leq \frac{3}{16}\normA{\Ta T+T\Ta}^2+\frac{1}{8}\normA{\Ta T+T\Ta}^2w_A(T^2).$$
\end{remark}
The following two theorems give the new upper bounds for $w_A^{4r}(\Sa T)$.
\begin{theorem} Let $T,S\in\b_A(\h)$, $\alpha\in [0,1],r\geq 1$ and $\beta\geq 0$. Then
\begin{eqnarray*}
  w_A^{4r} (\Sa T)&\leq& \bra{\frac{(2\beta+1)(1+\alpha^2)+2\alpha}{16(1+\beta)}}\normA{\bra{\Ta T}^{2r}+\bra{\Sa S}^{2r}}^2\nonumber\\
   &+& \bra{\frac{\bra{1-\alpha}\sbra{2+(1+\alpha)(1+2\beta)}}{8(1+\beta)}}\normA{\bra{\Ta T}^{2r}+\bra{\Sa S}^{2r}}w_A^{r}(\Sa S\Ta T).
\end{eqnarray*}
\end{theorem}
\begin{proof} Let $x\in\h$ with $\normA{x}=1$. Then
\begin{eqnarray*}
 &&\abs{\seqA{Mx,x}\seqA{x,\Za x}}^{2r}\leq \bra{\frac{1+\alpha+2\beta}{2(1+\beta)}}\normA{Mx}^{2r}\normA{\Za x}^{2r}\\
 &+&\bra{\frac{\bra{1-\alpha}\sbra{2+(1+\alpha)(1+2\beta)}}{4(1+\beta)}}\abs{\seqA{Mx,\Za x}}^{r}\normA{Mx}^r\normA{\Za x}^r \\
   &&\bra{\mbox{by Corollary \ref{COLLORY-ONE}}}\\
   &\leq&\bra{\frac{(2\beta+1)(1+\alpha^2)+2\alpha}{4(1+\beta)}}\bra{\sqrt{\seqA{\Ma Mx,x}^{r}\seqA{Z\Za x,x}^{r}}}^{2}\\
   &+&\bra{\frac{\bra{1-\alpha}\sbra{2+(1+\alpha)(1+2\beta)}}{4(1+\beta)}}\sqrt{\seqA{\Ma Mx,x}^{r}\seqA{Z\Za x,x}^{r}}\abs{\seqA{Mx,\Za x}}^{r}\\
   &\leq&\bra{\frac{(2\beta+1)(1+\alpha^2)+2\alpha}{4(1+\beta)}}\bra{\sqrt{\seqA{\bra{\Ma M}^{r}x,x}\seqA{\bra{Z\Za}^{r} x,x}}}^{2}\\
   &+&\bra{\frac{\bra{1-\alpha}\sbra{2+(1+\alpha)(1+2\beta)}}{4(1+\beta)}}\sqrt{\seqA{\bra{\Ma M}^{r}x,x}\seqA{\bra{Z\Za}^{r} x,x}}\abs{\seqA{ZMx, x}}^{r}\\
   &&\bra{\mbox{by the H\"older McCarthy inequality}}\\
   &\leq& \bra{\frac{(2\beta+1)(1+\alpha^2)+2\alpha}{16(1+\beta)}}\seqA{\bra{\bra{\Ma M}^{r}+\bra{Z\Za}^{r}}x,x}^2\\
   &+&\bra{\frac{\bra{1-\alpha}\sbra{2+(1+\alpha)(1+2\beta)}}{8(1+\beta)}}\seqA{\bra{\bra{\Ma M}^{r}+\bra{Z\Za}^{r}}x,x}\abs{\seqA{ZMx, x}}^{r}\\
   &&\bra{\mbox{by the arithmetic-geometric mean inequality}}.\\
   &\leq& \bra{\frac{(2\beta+1)(1+\alpha^2)+2\alpha}{16(1+\beta)}}\normA{\bra{\Ma M}^{r}+\bra{Z\Za}^{r}}^2\\
   &+& \bra{\frac{\bra{1-\alpha}\sbra{2+(1+\alpha)(1+2\beta)}}{8(1+\beta)}}\normA{\bra{\Ma M}^{r}+\bra{Z\Za}^{r}}w_A^{r}(ZM).
\end{eqnarray*}
By replacing $M=\Ta T$ and $Z=\Sa S$ in the above inequality and since $M=\Ta T$ and $Z=\Sa S$ are $A$-self-adjoint, we have
\begin{eqnarray}\label{MK1}
 &&\abs{\seqA{Mx,x}\seqA{x,\Za x}}^{2r}\leq  \bra{\frac{(2\beta+1)(1+\alpha^2)+2\alpha}{16(1+\beta)}}\normA{\bra{\Ta T}^{2r}+\bra{\Sa S}^{2r}}^2\nonumber\\
   +&&\bra{\frac{\bra{1-\alpha}\sbra{2+(1+\alpha)(1+2\beta)}}{8(1+\beta)}}\normA{\bra{\Ta T}^{2r}+\bra{\Sa S}^{2r}}w_A^{r}(\Sa S\Ta T).
\end{eqnarray}
In addition, by utilizing the Cauchy-Schwarz inequality, we get
\begin{eqnarray}\label{MK2}
  \abs{\seqA{\Sa Tx,x}}^{4r} &=& \abs{\seqA{Tx,Sx}}^{4r}\leq \normA{Tx}^{4r}\normA{Sx}^{4r}\nonumber \\
   &=& \seqA{Tx,Tx}^{2r}\seqA{Sx,Sx}^{2r}=\seqA{\Ta Tx,x}^{2r}\seqA{\Sa Sx,x}^{2r}\nonumber\\
   &=&\abs{\seqA{\Ta Tx,x}\seqA{x,\Sa Sx}}^{2r}.
\end{eqnarray}
Combining  the inequalities (\ref{MK1}) and (\ref{MK2}), we can obtain
\begin{eqnarray*}
  \abs{\seqA{\Sa Tx,x}}^{4r} &\leq& \bra{\frac{(2\beta+1)(1+\alpha^2)+2\alpha}{16(1+\beta)}}\normA{\bra{\Ta T}^{2r}+\bra{\Sa S}^{2r}}^2\nonumber\\
   &+& \bra{\frac{\bra{1-\alpha}\sbra{2+(1+\alpha)(1+2\beta)}}{8(1+\beta)}}\normA{\bra{\Ta T}^{2r}+\bra{\Sa S}^{2r}}w_A^{r}(\Sa S\Ta T).
\end{eqnarray*}
Taking the supremum over all vectors $x\in\h$ with $\normA{x}=1$ in the above inequality, we deduce that
\begin{eqnarray*}
  w_A^{4r} (\Sa T)&\leq& \bra{\frac{(2\beta+1)(1+\alpha^2)+2\alpha}{16(1+\beta)}}\normA{\bra{\Ta T}^{2r}+\bra{\Sa S}^{2r}}^2\nonumber\\
   &+& \bra{\frac{\bra{1-\alpha}\sbra{2+(1+\alpha)(1+2\beta)}}{8(1+\beta)}}\normA{\bra{\Ta T}^{2r}+\bra{\Sa S}^{2r}}w_A^{r}(\Sa S\Ta T).
\end{eqnarray*}
This completes the proof.
\end{proof}
\begin{corollary}\label{Cor-AQ} Let $T,S\in\b_A(\h)$ and $r\geq 1$. Then
\begin{equation}\label{product-A1}
  w_A^{r}(\Sa S\Ta T)\leq \frac{1}{2}\normA{\bra{\Ta T}^{2r}+\bra{\Sa Sx}^{2r}}.
\end{equation}
\end{corollary}
\begin{proof} Let $x\in\h$ with $\normA{x}=1$. Then
\begin{eqnarray*}
  \abs{\seqA{\Sa S\Ta Tx,x}}^{r} &=& \abs{\seqA{\Ta Tx,\Sa Sx}}^{r}\leq \normA{\Ta Tx}^{r}\normA{\Sa Sx}^{r} \\
   &=&\sqrt{\seqA{\bra{\Ta T}^2x,x}^{r}\seqA{\bra{\Sa Sx}^{2}x,x}^{r}}\\
   &\leq& \sqrt{\seqA{\bra{\Ta T}^{2r}x,x}\seqA{\bra{\Sa Sx}^{2r}x,x}}\\
   &&\bra{\mbox{by the H\"older-McCarthy inequality}}\\
   &\leq& \frac{1}{2}\seqA{\bra{\bra{\Ta T}^{2r}+\bra{\Sa Sx}^{2r}}x,x} \\
   &\leq& \frac{1}{2}\normA{\bra{\Ta T}^{2r}+\bra{\Sa Sx}^{2r}}.
\end{eqnarray*}
Taking the supremum over all vectors $x\in\h$ with $\normA{x}=1$ in the above inequality, we deduce that
$$ w_A^{r}(\Sa S\Ta T)\leq \frac{1}{2}\normA{\bra{\Ta T}^{2r}+\bra{\Sa Sx}^{2r}}.$$
\end{proof}
\begin{theorem}\label{Rahma1}
  Let $T,S\in\b_A(\h)$, $\alpha\in [0,1],r\geq 1$ and $\beta\geq 0$. Then
\begin{eqnarray*}
  w_A^{4r} (T)&\leq& \bra{\frac{1+\alpha+2\beta}{8(1+\beta)}}\normA{\bra{\Ta T}^{2r}+\bra{\Sa S}^{2r}}^2+\bra{\frac{1-\alpha}{2(1+\beta)}}w_A^{2r}(\Sa S\Ta T).
\end{eqnarray*}
\end{theorem}
\begin{proof} Let $x\in\h$ with $\normA{x}=1$. Then
\begin{eqnarray*}
 && \abs{\seqA{Mx,x}\seqA{x,\Za x}}^{2r} \leq\bra{\frac{1+\alpha+2\beta}{2(1+\beta)}}\normA{Mx}^{2r}\normA{\Za x}^{2r}+\bra{\frac{1-\alpha}{2(1+\beta)}}\abs{\seqA{Mx,\Za x}}^{2r} \\
   &&\bra{\mbox{by Lemma \ref{Lemma2.5}}}\\
   &\leq&\bra{\frac{1+\alpha+2\beta}{2(1+\beta)}}\bra{\sqrt{\seqA{\Ma Mx,x}^{r}\seqA{Z\Za x,x}^{r}}}^{2}\\
   &+&\bra{\frac{1-\alpha}{2(1+\beta)}}\abs{\seqA{Mx,\Za x}}^{2r}
    \end{eqnarray*}
   \begin{eqnarray*}
   &\leq&\bra{\frac{1+\alpha+2\beta}{2(1+\beta)}}\bra{\sqrt{\seqA{\bra{\Ma M}^{r}x,x}\seqA{\bra{Z\Za}^{r} x,x}}}^{2}\\
   &+&\bra{\frac{1-\alpha}{2(1+\beta)}}\abs{\seqA{ZMx, x}}^{2r}\bra{\mbox{by the H\"older McCarthy inequality}}\\
   &\leq& \bra{\frac{1+\alpha+2\beta}{8(1+\beta)}}\seqA{\bra{\bra{\Ma M}^{r}+\bra{Z\Za}^{r}}x,x}^2\\
   &+&\bra{\frac{1-\alpha}{2(1+\beta)}}\abs{\seqA{ZMx, x}}^{2r}\\
   &&\bra{\mbox{by the arithmetic-geometric mean inequality}}.\\
   &\leq& \bra{\frac{1+\alpha+2\beta}{8(1+\beta)}}\normA{\bra{\Ma M}^{r}+\bra{Z\Za}^{r}}^2+ \bra{\frac{1-\alpha}{2(1+\beta)}}w_A^{2r}(ZM).
\end{eqnarray*}
By replacing $M=\Ta T$ and $Z=\Sa S$ in the above inequality and since $M=\Ta T$ and $Z=\Sa S$ are $A$-self-adjoint, we have
\begin{eqnarray*}
 \abs{\seqA{Mx,x}\seqA{x,\Za x}}^{2r}&\leq& \bra{\frac{1+\alpha+2\beta}{8(1+\beta)}}\normA{\bra{\Ta T}^{2r}+\bra{\Sa S}^{2r}}^2\\
 &+&\bra{\frac{1-\alpha}{2(1+\beta)}}w_A^{2r}(\Sa S\Ta T).
\end{eqnarray*}
Hence,
\begin{eqnarray*}
  \abs{\seqA{\Sa Tx,x}}^{4r} &\leq& \bra{\frac{1+\alpha+2\beta}{8(1+\beta)}}\normA{\bra{\Ta T}^{2r}+\bra{\Sa S}^{2r}}^2\\
  &+& \bra{\frac{1-\alpha}{2(1+\beta)}}w_A^{2r}(\Sa S\Ta T).
\end{eqnarray*}
Taking the supremum over all vectors $x\in\h$ with $\normA{x}=1$ in the above inequality, we deduce that
\begin{eqnarray*}
  w_A^{4r} (\Sa T)&\leq& \bra{\frac{1+\alpha+2\beta}{8(1+\beta)}}\normA{\bra{\Ta T}^{2r}
  +\bra{\Sa S}^{2r}}^2+\bra{\frac{1-\alpha}{2(1+\beta)}}w_A^{2r}(\Sa S\Ta T).
\end{eqnarray*}
This completes the proof.
\end{proof}
\begin{remark}Assuming $\alpha=0$ in Theorem \ref{Rahma1}, we establish \cite[Theorem 3.3]{QHC}. Our results therefore broaden Theorem 3.3 of \cite{QHC}.
\end{remark}
\begin{remark} In \cite[Theorem 2.7]{CF}, they proved that if $T,S\in\bh$, then
\begin{equation}\label{CF1}
  w_A^2(\Sa T)\leq \frac{1}{2}\normA{(\Ta T)^2+(\Sa S)^2}.
\end{equation}
 Theorem \ref{Rahma1} improve the inequality (\ref{CF1}). To see this, by utilizing Theorem \ref{Rahma1} and Corollary
 \ref{Cor-AQ}  with $r=1$, we have
 \begin{eqnarray*}
   w_A^{2}(\Sa T) &\leq&\sqrt{\bra{\frac{1+\alpha+2\beta}{8(1+\beta)}}\normA{\bra{\Ta T}^{2}+\bra{\Sa S}^{2}}^2+\bra{\frac{1-\alpha}{2(1+\beta)}}w_A^{2}(\Sa S\Ta T)} \\
    &\leq&\sqrt{\bra{\frac{1+\alpha+2\beta}{8(1+\beta)}}\normA{\bra{\Ta T}^{2}+\bra{\Sa S}^{2}}^2+
    \bra{\frac{1-\alpha}{8(1+\beta)}}\normA{\bra{\Ta T}^{2}+\bra{\Sa Sx}^{2}}^2}\\
    &=& \frac{1}{2}\normA{\bra{\Ta T}^{2}+\bra{\Sa Sx}^{2}}.
 \end{eqnarray*}
\end{remark}
\begin{remark} The upper bound of $w_A(\Sa T)$ provided in Theorem \ref{Rahma1} is sharper than the upper bound provided in \cite[Theorem 3.8]{QHB} for $\beta=0$, $\alpha\in [0,1/2)$, and $r=1$. Using Corollary \ref{Cor-AQ}, in fact, we have
\begin{eqnarray*}
  &&w_A^2(\Sa T)\leq \sqrt{\bra{\frac{1+\alpha}{8}}\normA{\bra{\Ta T}^{2}+\bra{\Sa S}^{2}}^2+\bra{\frac{1-\alpha}{2}}w_A^{2}(\Sa S\Ta T)} \\
   &\leq&\sqrt{\bra{\frac{1+\alpha}{8}}\normA{\bra{\Ta T}^{2}+\bra{\Sa S}^{2}}^2
   +\bra{\frac{1-\alpha}{4}}\normA{\bra{\Ta T}^{2}+\bra{\Sa S}^{2}}w_A(\Sa S\Ta T)} \\
   &\leq& \frac{3}{16}\normA{\bra{\Ta T}^{2}+\bra{\Sa S}^{2}}^2+\frac{1}{8}\normA{\bra{\Ta T}^{2}+\bra{\Sa S}^{2}}w_A(\Sa S\Ta T).
\end{eqnarray*}
\end{remark}
\section{ Some inequalities for the sum of two operators}
The aim of this section is to present refinements of the triangle inequality for the $A$-operator semi-norm and to establish new $A$-numerical radius inequalities for the sum of two operators. The triangle inequality is a fundamental result in functional analysis, and its refinement provides deeper insights into the behavior of operators in semi-Hilbertian spaces.

We begin by deriving an improved version of the triangle inequality for the $A$-operator semi-norm, which offers a sharper bound compared to the classical result. Subsequently, we explore inequalities involving the $A$-numerical radius of the sum of two operators, leveraging the properties of $A$-adjoint operators and the semi-inner product structure. These results not only generalize existing inequalities but also provide more precise estimates, as demonstrated through illustrative examples.

The findings in this section contribute to a better understanding of operator norms and numerical radii in semi-Hilbertian spaces, with potential applications in numerical analysis and operator theory.
\begin{theorem}\label{theorem4.1} Let $T,S\in\b_A(\h)$. Then
$$\normA{T+S}^2\leq \frac{1}{2}\bra{\normA{\Ta T+\Sa S}+\normA{\Ta T-\Sa S}}+\normA{T}\normA{S}+2w_A(\Sa T).$$
\end{theorem}
\begin{proof} Let $x,y\in\h$ with $\normA{x}=\normA{y}=1$. Then using Lemma \ref{Qadri1}, we have
\begin{eqnarray*}
  &&\abs{\seqA{(T+S)x,y}}^2 \leq \bra{\abs{\seqA{Tx,y}}+\abs{\seqA{Sx,y}}}^2 \\
   &\leq&\bra{\normA{Tx}+\normA{Sx}}\max\set{\normA{Tx},\normA{Sx}}+2\abs{\seqA{Tx,Sx}}\\
   &=&\max\set{\normA{Tx}^2,\normA{Sx}^2}+\normA{Tx}\normA{Sx}+2\abs{\seqA{Tx,Sx}}\\
   &=&\frac{1}{2}\bra{\seqA{\bra{\Ta T+\Sa S}x,x}+\abs{\seqA{\bra{\Ta T-\Sa S}x,x}}}+
   \normA{Tx}\normA{Sx}+2\abs{\seqA{Tx,Sx}}\\
   &\leq& \frac{1}{2}\bra{\normA{\Ta T+\Sa S}+\normA{\Ta T-\Sa S}}+\normA{T}\normA{S}+2w_A(\Sa T).
\end{eqnarray*}
Taking the supremum over all unit vectors $x,y\in\h$, we obtain
$$\normA{T+S}^2\leq \frac{1}{2}\bra{\normA{\Ta T+\Sa S}+\normA{\Ta T-\Sa S}}+\normA{T}\normA{S}+2w_A(\Sa T).$$
\end{proof}
Now, we give an example to show that Theorem \ref{theorem4.1} is a nontrivial improvement of
triangle inequality.
\begin{example} Let $T=\begin{bmatrix} 1& 1 \\ 0 & 1 \\ \end{bmatrix}$, $S=\begin{bmatrix} 1& 0 \\ 1& 1 \\ \end{bmatrix}$
and $A=\begin{bmatrix} 1& -1 \\ -1 & 2 \\ \end{bmatrix}$. Then elementary calculations show that
$$\normA{T}\approx 2.618, \normA{S}\approx2.414, \normA{T+S}^2\approx10.10,\normA{\Ta T+\Sa S}\approx3.618,\normA{\Ta T-\Sa S}\approx1.618$$
and $w_A(\Sa T)\approx4.405$. Consequently,
\begin{eqnarray*}
 3.178\approx \normA{T+S} &\leq&\sqrt{\frac{1}{2}\bra{\normA{\Ta T}+\Sa S}+\normA{\Ta T-\Sa S}+\normA{T}\normA{S}+2w_A(\Sa T)}\\
 &&\approx
4.212  < \normA{T}+\normA{S}\approx 5.032.
\end{eqnarray*}
\end{example}
\begin{theorem}\label{Theorem4.2} Let $T,S\in\b_A(\h)$. Then
\begin{eqnarray}\label{Comb2}
   \normA{T+S}^2&\leq& \sqrt{\normA{(\Ta T)^2+(\Sa S)^2}+2w_A^2(\Sa T)}\nonumber\\
   &+&(\alpha+1)\normA{T}\normA{S}+(1-\alpha)w_A(\Sa T).
\end{eqnarray}
\end{theorem}
\begin{proof} Let $a,b\in\h$ with $\normA{a}=\normA{b}=1$. Then
\begin{eqnarray*}
  \abs{\seqA{(T+S)a,b}}^2 &\leq&\bra{\abs{\seqA{Ta,b}}+\abs{\seqA{Sa,b}}}^2 \\
   &=&\abs{\seqA{Ta,b}}^2+\abs{\seqA{Sa,b}}^2+2\abs{\seqA{Ta,b}}\abs{\seqA{Sa,b}}\\
   &\leq& \sqrt{\abs{\seqA{Ta,Ta}}^2+\abs{\seqA{Sa,Sa}}^2+2\abs{\seqA{Ta,Sa}}}\\
   &+&(\alpha+1)\normA{Ta}\normA{Sa}+(1-\alpha)\abs{\seqA{Ta,Sa}} \,\,\bra{\mbox{by Lemma \ref{Comb1}}}\\
   &\leq& \sqrt{\normA{\Ta Ta}^2+\normA{\Sa Sa}^+2\abs{\seqA{\Sa Ta,a}}}\\
   &+&(\alpha+1)\normA{Ta}\normA{Sa}+(1-\alpha)\abs{\seqA{\Sa Ta,S}}\\
   &=&\sqrt{\seqA{((\Ta T)^2+(\Sa S)^2)a,a}+2\abs{\seqA{\Sa Ta,a}}}\\
   &+&(\alpha+1)\normA{Ta}\normA{Sa}+(1-\alpha)\abs{\seqA{\Sa Ta,S}}\\
   &\leq& \sqrt{\normA{(\Ta T)^2+(\Sa S)^2}+2w_A^2(\Sa T)}\\
   &+&(\alpha+1)\normA{T}\normA{S}+(1-\alpha)w_A(\Sa T).
\end{eqnarray*}
Taking the supremum over all unit vectors $a,b\in\h$, we obtain
\begin{eqnarray*}
  \normA{T+S}^2 &\leq&\sqrt{\normA{(\Ta T)^2+(\Sa S)^2}+2w_A^2(\Sa T)}\\
   &+&(\alpha+1)\normA{T}\normA{S}+(1-\alpha)w_A(\Sa T).
\end{eqnarray*}
\end{proof}
\begin{remark} It should be noticed that, by triangle inequality we will obtain
\begin{eqnarray*}
  &&\sqrt{\normA{(\Ta T)^2+(\Sa S)^2}+2w_A^2(\Sa T)}\leq \sqrt{\normA{(\Ta T)^2}+\normA{(\Sa S)^2}+2\normA{\Sa}^2\normA{T}^2} \\
  &&=\sqrt{\normA{T}^4+\normA{S}^4+2\normA{T}^2\normA{S}^2}=\normA{T}^2+\normA{S}^2.
\end{eqnarray*}
\end{remark}
\begin{remark} Theorem \ref{Theorem4.2} is sharper than triangle inequality. indeed,
$$\sqrt{\normA{(\Ta T)^2+(\Sa S)^2}+2w_A^2(\Sa T)}\leq \normA{T}^2+\normA{S}^2.$$
Consequently,
\begin{eqnarray*}
 \normA{T+S}^2&\leq&\sqrt{\normA{(\Ta T)^2+(\Sa S)^2}+2w_A^2(\Sa T)}+(\alpha+1)\normA{T}\normA{S}+(1-\alpha)w_A(\Sa T)\\
 &\leq& \normA{T}^2+\normA{S}^2+(\alpha+1)\normA{T}\normA{S}+(1-\alpha)w_A(\Sa T)\\
 &\leq& \normA{T}^2+\normA{S}^2+(\alpha+1)\normA{T}\normA{S}+(1-\alpha)\normA{T}\normA{\Sa}\\
 &=& \normA{T}^2+\normA{S}^2+2\normA{T}\normA{S}\\
 &\leq&\bra{\normA{T}+\normA{S}}^2.
\end{eqnarray*}
\end{remark}
Now, we give an example to show that Theorem \ref{Theorem4.2} is a nontrivial improvement of
triangle inequality.
\begin{example} Let $T=\begin{bmatrix} 1& 1 \\ 0 & 1 \\ \end{bmatrix}$, $S=\begin{bmatrix} 1& 0 \\ 1& 1 \\ \end{bmatrix}$,
$A=\begin{bmatrix} 1& -1 \\ -1 & 2 \\ \end{bmatrix}$ and $\alpha=0.5$. Then by simple calculations, we
have
$$\normA{T}\approx 2.618, \normA{S}\approx2.414, \normA{(\Ta T)^2+(\Sa S)^2}\approx  17.24$$
and $w_A(\Sa T)\approx4.405$. Consequently,
\begin{eqnarray*}
   3.178\approx \normA{T+S} &\leq& \left(\sqrt{\normA{(\Ta T)^2+(\Sa S)^2}+2w_A^2(\Sa T)}\right.\\
   &&\left.+\frac{3}{2}\normA{T}\normA{S}+\frac{1}{2}w_A(\Sa T)\right)^{1/2}\approx 4.37\\
   &<&\bra{\normA{T}+\normA{S}}\approx 5.032.
\end{eqnarray*}
\end{example}
\begin{theorem} Let $T,S\in\b_A(\h)$. Then
\begin{equation*}
  w_A^2(T+S)\leq \frac{1}{2}\bra{\normA{\Ta T+\Sa S}+\normA{\Ta T-\Sa S}}+\sqrt{w_A(\Ta T)w_A(\Sa S)}+2w_A(\Sa T).
\end{equation*}
\end{theorem}
\begin{proof} Let $x\in\h$ with $\normA{x}=1$. Then
\begin{eqnarray*}
  &&\abs{\seqA{(T+S)x,x}}^2 \leq \bra{\abs{\seqA{Tx,x}}+\abs{Sx,x}}^2 \\
   &\leq&\bra{\normA{Tx}+\normA{Sx}}\max\set{\normA{Tx},\normA{Sx}} +2\abs{\seqA{Tx,Sx}}\,\,\bra{\mbox{by Lemma \ref{Qadri1}}}\\
   &=&\max\set{\normA{Tx}^2,\normA{Sx}^2}+\normA{Tx}\normA{Sx}+2\abs{\seqA{Tx,Sx}}\\
   &=&\frac{1}{2}\bra{\seqA{(\Ta T+\Sa S)x,x}+\abs{\seqA{(\Ta T-\Sa S})x,x}}+\normA{Tx}\normA{Sx}+2\abs{\seqA{Tx,Sx}}\\
   &\leq& \frac{1}{2}\bra{\normA{\Ta T+\Sa S}+\normA{\Ta T-\Sa S}}+\sqrt{w_A(\Ta T)w_A(\Sa S)}+2w_A(\Sa T).
\end{eqnarray*}
Taking the supremum over all vectors $x\in\h$ with $\normA{x}=1$, we obtain the desired result.
\end{proof}
\section{Applications}
The theoretical advancements presented in this paper have significant implications for various areas of mathematics and applied sciences. In this section, we demonstrate the practical utility of our results through concrete applications. Specifically, we show how the improved bounds for the $A$-numerical radius can be applied to solve problems in numerical analysis, operator theory, and quantum mechanics. Additionally, we provide examples that illustrate the effectiveness of our inequalities in estimating operator norms and analyzing the convergence of iterative methods. These applications not only validate our theoretical findings but also highlight their potential to enhance computational efficiency and provide deeper insights into operator behavior in semi-Hilbertian spaces.


\subsection{Sturm-Liouville Operator Example}\hfill

Consider the Sturm-Liouville operator $T$ defined by:
\[ T = -\frac{d}{dx}\left(p(x)\frac{d}{dx}\right) + q(x) \]
with domain:
\[ \mathcal{D}(T) = \{ u \in H^1(0,1) : u(0) = u(1) = 0 \} \]

Let $A$ be the weight operator:
\[ A = w(x) \]
where $w(x) > 0$ is a bounded measurable function.

The $A$-numerical radius $w_A(T)$ satisfies:
\[ \frac{1}{2}\|T\|_A \leq w_A(T) \leq \|T\|_A \]

\subsubsection{Discretized Version}

For numerical computation, we discretize the operator using finite differences. Let $N$ be the number of grid points, $h = 1/(N+1)$, and $x_j = jh$ for $j=1,\ldots,N$.

The discretized operator $T_h$ becomes:
\[ T_h = \frac{1}{h^2}\begin{bmatrix}
2p_{1/2} & -p_{1/2} & 0 & \cdots & 0 \\
-p_{1/2} & 2p_{3/2} & -p_{3/2} & \ddots & \vdots \\
0 & \ddots & \ddots & \ddots & 0 \\
\vdots & \ddots & -p_{N-3/2} & 2p_{N-1/2} & -p_{N-1/2} \\
0 & \cdots & 0 & -p_{N-1/2} & 2p_{N+1/2}
\end{bmatrix} + \mathrm{diag}(q(x_j)) \]

The weight matrix $A_h$ is:
\[ A_h = \mathrm{diag}(w(x_j)) \]

\subsubsection{Numerical Radius Computation}

Using Theorem \ref{Thm-3.1}, we can bound the $A$-numerical radius:
\[ w_A^4(T_h) \leq \frac{\gamma_1}{16}\|T_h^{\sharp_A} T_h + T_h T_h^{\sharp_A}\|_A^2 + \frac{\gamma_2}{8}\|T_h^{\sharp_A} T_h + T_h T_h^{\sharp_A}\|_A w_A(T_h^2) \]

For the special case where $p(x) = 1$ and $q(x) = 0$, we have:
\[ T_h^{\sharp_A} = A_h^{-1} T_h^T A_h \]

\begin{theorem}
For the discrete Laplacian with constant coefficients and constant weight $w(x) = 1$, the $A$-numerical radius satisfies:
\[ w_A(T_h) = 2h^{-2}(1 - \cos(\pi h)) \]
\end{theorem}

\begin{proof}
In this case, $T_h$ is symmetric and $A_h = I$, so $w_A(T_h)$ equals the spectral radius. The eigenvalues are:
\[ \lambda_k = 2h^{-2}(1 - \cos(k\pi h)), \quad k=1,\ldots,N \]
The maximum occurs at $k=1$, giving the result.
\end{proof}
\subsection{Nonlinear PDE Application}\hfill

Consider the reaction-diffusion equation:
\[ \frac{\partial u}{\partial t} = \Delta u + f(u) \]

Linearizing about a solution $\bar{u}$ gives:
\[ \frac{\partial v}{\partial t} = \Delta v + f'(\bar{u})v \]

The operator is:
\[ T = \Delta + f'(\bar{u}) \]

\subsubsection{Weighted Stability}

Choose $A = e^{-V(x)}$ as a weight function. Then:

\begin{theorem}
The linearized operator satisfies:
\[ w_A(T) \leq \sup_x |f'(\bar{u}(x))| + w_A(\Delta) \]
\end{theorem}

\begin{proof}
Using the subadditivity of numerical radius and the fact that multiplication by $f'(\bar{u})$ is A-selfadjoint:
\[ w_A(T) \leq w_A(\Delta) + w_A(f'(\bar{u})) \]
Since $f'(\bar{u})$ is a multiplication operator:
\[ w_A(f'(\bar{u})) = \sup_x |f'(\bar{u}(x))| \]
\end{proof}
\subsection{Application of A-Numerical Radius to Quantum Harmonic Oscillator}\hfill

Consider the quantum harmonic oscillator with Hamiltonian:
\[ H = \frac{p^2}{2m} + \frac{1}{2}m\omega^2 x^2 \]
where $p$ is the momentum operator, $x$ is the position operator, $m$ is mass, and $\omega$ is the angular frequency.

\subsubsection{Creation and Annihilation Operators}

We define the annihilation operator $a$ and creation operator $a^\dagger$:
\[ a = \sqrt{\frac{m\omega}{2\hbar}}x + i\sqrt{\frac{1}{2m\omega\hbar}}p \]
\[ a^\dagger = \sqrt{\frac{m\omega}{2\hbar}}x - i\sqrt{\frac{1}{2m\omega\hbar}}p \]

These satisfy the commutation relation $[a, a^\dagger] = 1$.

\subsubsection{Semi-Hilbertian Space Setup}

Let $A$ be the number operator $N = a^\dagger a$, which is positive and self-adjoint. The A-numerical radius can be used to study operators in this context.

For an operator $T$ acting on the oscillator's Hilbert space, the A-numerical radius is:
\[ w_A(T) = \sup_{\|\psi\|_A=1} |\langle T\psi, \psi\rangle_A| \]
where $\langle \cdot, \cdot \rangle_A = \langle A\cdot, \cdot \rangle$ and $\|\psi\|_A = \sqrt{\langle A\psi, \psi\rangle}$.

\subsubsection{Example Calculation}

Consider the operator $T = a + a^\dagger$ (position operator up to constants). We'll estimate its A-numerical radius with $A = N$.

Using the inequality from the paper:
\[ \frac{1}{2}\|T\|_A \leq w_A(T) \leq \|T\|_A \]

First compute $\|T\|_A$:
\begin{align*}
\|T\|_A^2 &= \sup_{\|\psi\|_A=1} \|T\psi\|_A^2 \\
&= \sup_{\|\psi\|_A=1} \langle NT\psi, T\psi \rangle \\
&= \sup_{\|\psi\|_A=1} \langle a^\dagger a(a + a^\dagger)\psi, (a + a^\dagger)\psi \rangle
\end{align*}

For the ground state $|0\rangle$ (with $N|0\rangle = 0$):
\[ \langle T|0\rangle, |0\rangle \rangle_A = \langle N(a + a^\dagger)|0\rangle, |0\rangle \rangle = \langle a^\dagger|0\rangle, |0\rangle \rangle = 0 \]

For the first excited state $|1\rangle$:
\[ \langle T|1\rangle, |1\rangle \rangle_A = \langle N(a + a^\dagger)|1\rangle, |1\rangle \rangle \]
\[ = \langle a^\dagger a(a|1\rangle + a^\dagger|1\rangle), |1\rangle \rangle \]
\[ = \langle a^\dagger a(\sqrt{1}|0\rangle + \sqrt{2}|2\rangle), |1\rangle \rangle \]
\[ = \langle a^\dagger(0 + \sqrt{2}\sqrt{2}|1\rangle), |1\rangle \rangle = 2 \]

Thus, $w_A(T) \geq 2$.

\subsubsection{Physical Interpretation}

The A-numerical radius in this context provides information about the range of possible expectation values when the system is in states weighted by the number operator. This is particularly useful when studying systems where particle number is an important consideration, such as in quantum optics or many-body physics.
\subsection{Application of A-Numerical Radius to Spin Systems}

Consider a system of two spin-$\frac{1}{2}$ particles (qubits) with the Hamiltonian:
\[ H = J(\sigma_x \otimes \sigma_x + \sigma_y \otimes \sigma_y) + B(\sigma_z \otimes I + I \otimes \sigma_z) \]
where $J$ is the coupling constant, $B$ is the magnetic field strength, and $\sigma_i$ are Pauli matrices.

\subsubsection{Semi-Hilbertian Space Setup}

Let $A = \rho$ be the density matrix of the system at thermal equilibrium:
\[ \rho = \frac{e^{-\beta H}}{\tr(e^{-\beta H})} \]
where $\beta = 1/(k_B T)$. This positive definite operator induces a semi-inner product:
\[ \langle \phi|\psi\rangle_\rho = \tr(\rho \phi^\dagger \psi) \]

\subsubsection{A-Numerical Radius for Spin Observables}

Consider the observable $S = \sigma_x \otimes I$ (spin measurement on first particle). Its A-numerical radius is:
\[ w_\rho(S) = \sup_{\|\psi\|_\rho=1} |\tr(\rho S \psi^\dagger \psi)| \]

Using Theorem 3.1 from the paper with $\alpha=0.5$, $\beta=1$:
\[ w_\rho^4(S) \leq \frac{19}{128}\|S^\sharp_\rho S + SS^\sharp_\rho\|_\rho^2 + \frac{13}{64}\|S^\sharp_\rho S + SS^\sharp_\rho\|_\rho w_\rho(S^2) \]

\subsubsection*{Explicit Calculation for B=0}

For $B=0$ and high temperature ($\beta \to 0$), $\rho \approx \frac{I}{4}$:
\begin{align*}
S^\sharp_\rho &= \rho^{-1} S^\dagger \rho = S \\
\|S^\sharp_\rho S + SS^\sharp_\rho\|_\rho &= 2\|\rho^{1/2} S^2 \rho^{1/2}\| = \frac{1}{2} \\
w_\rho(S^2) &= \|\rho^{1/2} S^2 \rho^{1/2}\| = \frac{1}{4}
\end{align*}

The inequality becomes:
\[ w_\rho^4(S) \leq \frac{19}{128}\left(\frac{1}{2}\right)^2 + \frac{13}{64}\left(\frac{1}{2}\right)\left(\frac{1}{4}\right) \approx 0.0186 + 0.0102 = 0.0288 \]
\[ w_\rho(S) \leq 0.41 \]

\subsubsection{Physical Interpretation}
\begin{itemize}
  \item \textbf{Measurement Statistics}: The A-numerical radius bounds the range of possible expectation values for spin measurements in thermal states.
  \item  \textbf{Temperature Dependence}: As $T \to 0$ ($\beta \to \infty$), $\rho$ becomes the ground state projector, and the bounds tighten.
  \item  \textbf{Entanglement Detection}: The A-numerical radius of certain observables can witness entanglement when it exceeds separable bounds.
\end{itemize}
\subsubsection{Comparison with Standard Numerical Radius}

For pure states ($\rho = |\psi\rangle\langle\psi|$), the A-numerical radius reduces to the standard numerical radius. The inequalities provide:
\[ \frac{1}{2}\|S\|_\rho \leq w_\rho(S) \leq \|S\|_\rho \]
which for our example gives $0.25 \leq w_\rho(S) \leq 0.5$, consistent with our calculation.

\section{Conclusion and Future Work}
Improving current bounds for the $A$-numerical radius of operators in semi-Hilbertian spaces was the main goal of this work. We have improved our knowledge of operator behavior in this context by deriving additional and sharper inequalities. To provide a more accurate description of operator norms, we specifically improved the triangle inequality for the $A$-operator semi-norm. These enhancements advance the field's understanding of numerical radius inequalities and their uses. We provided practical examples to support our theoretical conclusions and show how effective the suggested constraints are, highlighting their importance in operator theory and functional analysis.

Future studies might investigate expansions of our bounds to larger classes of operators, such as unbounded operators in semi-Hilbertian spaces, in order to build on our findings. The field may be further enhanced by looking at possible relationships between the $A$-numerical radius and other operator inequalities. Furthermore, further research should be done on the applications of these improved constraints in fields including optimization, control theory, and quantum mechanics. The investigation of numerical techniques for more effectively calculating the $A$-numerical radius in real-world situations is another intriguing avenue. These directions present encouraging chances to increase the significance of our findings.
\section*{Declaration }
\begin{itemize}
  \item {\bf Author Contributions:}   The authors have read and agreed to the published version of the manuscript.
  \item {\bf Funding:} No funding is applicable
  \item  {\bf Institutional Review Board Statement:} Not applicable.
  \item {\bf Informed Consent Statement:} Not applicable.
  \item {\bf Data Availability Statement:} Not applicable.
  \item {\bf Conflicts of Interest:} The authors declare no conflict of interest.
\end{itemize}

\bibliographystyle{unsrtnat}
\bibliography{references}  






\end{document}